\documentclass[final,3p,10pt,authoryear]{elsarticle}

\makeatletter
\def\ps@pprintTitle{%
 \let\@oddhead\@empty
 \let\@evenhead\@empty
 \def\@oddfoot{}%
 \let\@evenfoot\@oddfoot}
\makeatother

\usepackage{amssymb,amsthm,mathtools}
\usepackage[colorlinks]{hyperref}
\usepackage{natbib}

\DeclareRobustCommand{\brkbinom}{\genfrac\{\}{0pt}{}}
\usepackage{xcolor,dsfont,geometry}
\newtheorem{Theorem}{Theorem}
\newtheorem{Lemma}{Lemma}

\makeatletter \@addtoreset{equation}{section} \makeatother

\newcommand{\N}{\mathbb{N}}
\newcommand{\R}{\mathbb{R}}

\newcommand{\EE}{\mathbb{E}}

\newcommand{\bb}[1]{\boldsymbol{#1}}

\newcommand{\leqdef}{\vcentcolon=}

\newcommand{\blue}{\color{blue}}
\newcommand{\ind}{\mathds{1}}

\allowdisplaybreaks

\begin{document}

\begin{frontmatter}

    \title{General formulas for the central and non-central moments of the multinomial distribution}%

    \author[a1]{Fr\'ed\'eric Ouimet\texorpdfstring{\corref{cor1}\fnref{fn1}}{)}}%

    \address[a1]{California Institute of Technology, Pasadena, USA.}%

    \cortext[cor1]{Corresponding author}%
    \ead{ouimetfr@caltech.edu}%

    \begin{abstract}
        We present the first general formulas for the central and non-central moments of the multinomial distribution, using a combinatorial argument and the factorial moments previously obtained in \texorpdfstring{\cite{MR143299}}{Mosimann (1962)}. We use the formulas to give explicit expressions for all the non-central moments up to order~8 and all the central moments up to order~4. These results expand significantly on those in \texorpdfstring{\cite{Newcomer2008phd,Newcomer_et_al_2008_tech_report}}{Newcomer (2008) and Newcomer et al.\ (2008)}, where the non-central moments were calculated up to order~4.
    \end{abstract}

    \begin{keyword}
        multinomial distribution \sep higher moments \sep central moments \sep non-central moments
        \MSC[2010]{Primary : 62E15 Secondary : 60E05}
    \end{keyword}

\end{frontmatter}

\section{The multinomial distribution}\label{sec:intro}

    For any $d\in \N$, the $d$-dimensional (unit) simplex is defined by $\mathcal{S} \leqdef \big\{\bb{x}\in [0,1]^d : \sum_{i=1}^d x_i \leq 1\big\}$, and the probability mass function $\bb{k}\mapsto P_{\bb{k},m}(\bb{x})$ for $\bb{\xi} \leqdef (\xi_1,\xi_2,\dots,\xi_d) \sim \mathrm{Multinomial}\hspace{0.2mm}(m,\bb{x})$ is defined by
    \begin{equation}\label{eq:multinomial.probability}
        P_{\bb{k},m}(\bb{x}) \leqdef \frac{m!}{(m - \sum_{i=1}^d k_i)! \prod_{i=1}^d k_i!} \cdot (1 - \sum_{i=1}^d x_i)^{m - \sum_{i=1}^d k_i} \prod_{i=1}^d x_i^{k_i}, \quad \bb{k}\in \N_0^d \cap m\mathcal{S},
    \end{equation}
    where $m\in \N$ and $\bb{x}\in \mathcal{S}$.
    If $x_{d+1}\leqdef 1-\sum_{i=1}^d x_i$, then \eqref{eq:multinomial.probability} is just a reparametrization of $(\bb{\xi},1-\sum_{i=1}^d \xi_i)\sim \mathrm{Multinomial}\hspace{0.2mm}(m,(\bb{x},x_{d+1}))$ where $\sum_{i=1}^{d+1} x_i = 1$.
    In this paper, our main goal is to give general formulas for the non-central and central moments of \eqref{eq:multinomial.probability}, namely
    \begin{equation}\label{eq:main.goal}
        \EE\Big[\prod_{i=1}^d \xi_i^{p_i}\Big] \quad \text{and} \quad \EE\Big[\prod_{i=1}^d (\xi_i - \EE[\xi_i])^{p_i}\Big], \qquad p_1,p_2,\dots,p_d\in \N_0.
    \end{equation}
    We obtain the formulas using a combinatorial argument and the general expression for the factorial moments found in \cite{MR143299}, which we register in the lemma below.
    \begin{Lemma}[{\blue Factorial moments}]\label{lem:factorial.moments}
        Let $\bb{\xi}\sim \mathrm{Multinomial}\hspace{0.3mm}(m,\bb{x})$.
        Then, for all $k_1, k_2, \dots, k_d\in \N_0$,
        \begin{equation}\label{eq:factorial.moments}
            \EE\Big[\prod_{i=1}^d \xi_i^{(k_i)}\Big] = m^{(\sum_{i=1}^d k_i)} \prod_{i=1}^d x^{k_i},
        \end{equation}
        where $m^{(k)} \leqdef m (m-1) \dots (m-k+1)$ denotes the {\it $k$-th order falling factorial of $m$}.
    \end{Lemma}
    The formulas that we develop for the expectations in \eqref{eq:main.goal} will be used to compute explicitly all the non-central moments up to order~8 and all the central moments up to order~4, which expands on the third and fourth order non-central moments that were previously calculated in \cite[Appendix A.1]{Newcomer2008phd}.
    We should also mention that explicit formulas for several lower-order (mixed) cumulants were presented in \cite{MR33996} (see also \cite[p.37]{MR1429617}), but not for the moments.

\section{Motivation}\label{sec:motivation}

    To the best of our knowledge, general formulas for the central and non-central moments of the multinomial distribution have never been derived in the literature.
    The central moments can arise naturally, for example, when studying asymptotic properties, via Taylor expansions, of statistical estimators involving the multinomial distribution.
    For a given sequence of i.i.d.\ observations $\bb{X}_1,\bb{X}_2,\dots,\bb{X}_n$, two examples of such estimators are the Bernstein estimator for the cumulative distribution function
    \begin{equation}\label{eq:cdf.Bernstein.estimator}
        F_{n,m}^{\star}(\bb{x}) \leqdef \sum_{\bb{k}\in \N_0^d \cap m \mathcal{S}} \frac{1}{n} \sum_{i=1}^n \ind_{(-\bb{\infty},\frac{\bb{k}}{m}]}(\bb{X}_i) P_{\bb{k},m}(\bb{x}), \quad \bb{x}\in \mathcal{S}, ~~m,n \in \N,
    \end{equation}
    and the Bernstein estimator for the density function (also called smoothed histogram)
    \begin{equation}\label{eq:histogram.estimator}
        \hat{f}_{n,m}(\bb{x}) \leqdef \sum_{\bb{k}\in \N_0^d \cap (m-1) \mathcal{S}} \frac{m^d}{n} \sum_{i=1}^n \ind_{(\frac{\bb{k}}{m}, \frac{\bb{k} + 1}{m}]}(\bb{X}_i) P_{\bb{k},m-1}(\bb{x}), \quad \bb{x}\in \mathcal{S}, ~~m,n \in \N,
    \end{equation}
    over the $d$-dimensional simplex.
    Some of their asymptotic properties were investigated by \cite{MR0397977,MR0638651,MR0726014,MR0791719,MR0858109,MR1437794,MR1703623,MR1712051,MR1873330,MR1881846,MR1910059,MR2068610,MR2351744,MR2395599,MR2488150,MR2662607,MR2782409,MR2960952,MR2925964,MR3174309,MR3139345,MR3412755,MR3488598,MR3740722,MR3630225,MR3983257} when $d = 1$, by \cite{MR1293514} when $d = 2$, and by \cite{arXiv:2002.07758,arXiv:2006.11756} for all $d\geq 1$, using a local limit theorem from \cite{arXiv:2001.08512} for the multinomial distribution (see also \cite{MR0478288}).
    The estimator \eqref{eq:histogram.estimator} is a discrete analogue of the Dirichlet kernel estimator introduced by \cite{doi:10.2307/2347365} and studied theoretically in \cite{MR1685301,MR1718494,MR1742101,MR1985506} when $d = 1$ (among others), and in \cite{arXiv:2002.06956} for all $d\geq 1$.

\section{Results}\label{sec:main.results}

    First, we give a general formula of the non-central moments of the multinomial distribution in \eqref{eq:multinomial.probability}.

    \begin{Theorem}[{\blue Non-central moments}]\label{thm:non.central.moments}
        Let $\bb{\xi}\sim \mathrm{Multinomial}\hspace{0.3mm}(m,\bb{x})$.
        Then, for all $p_1, p_2, \dots, p_d\in \N_0$,
        \begin{equation}
            \EE\Big[\prod_{i=1}^d \xi_i^{p_i}\Big] = \sum_{k_1=0}^{p_1} \dots \sum_{k_d=0}^{p_d} m^{(\sum_{i=1}^d k_i)} \prod_{i=1}^d \brkbinom{p_i}{k_i} x_i^{k_i},
        \end{equation}
        where $\brkbinom{p}{k}$ denotes a Stirling number of the second kind (i.e., the number of ways to partition a set of $p$ objects into $k$ non-empty subsets).
    \end{Theorem}

    \begin{proof}
        We have the following well-known relation between the power $p\in \N_0$ of a number $x\in \R$ and the falling factorials of $x$:
        \begin{equation}
            x^p = \sum_{k=0}^p \brkbinom{p}{k} \, x^{(k)}.
        \end{equation}
        See, e.g., \cite[p.262]{MR1397498}.
        Apply this relation to every $\xi_i^{p_i}$ and use the linearity of the expectation to get
        \begin{equation}
            \EE\Big[\prod_{i=1}^d \xi_i^{p_i}\Big] = \sum_{k_1=0}^{p_1} \dots \sum_{k_d=0}^{p_d} \brkbinom{p_1}{k_1} \dots \brkbinom{p_d}{k_d} \, \EE\Big[\prod_{i=1}^d \xi_i^{(k_i)}\Big],
        \end{equation}
        The conclusion follows from Lemma~\ref{lem:factorial.moments}.
    \end{proof}

    We can now deduce a general formula for the central moments of the multinomial distribution.

    \begin{Theorem}[{\blue Central moments}]\label{thm:central.moments}
        Let $\bb{\xi}\sim \mathrm{Multinomial}\hspace{0.3mm}(m,\bb{x})$.
        Then, for all $p_1, p_2, \dots, p_d\in \N_0$,
        \begin{equation}
            \EE\Big[\prod_{i=1}^d (\xi_i - \EE[\xi_i])^{p_i}\Big] = \sum_{\ell_1=0}^{p_1} \dots \sum_{\ell_d=0}^{p_d} \sum_{k_1=0}^{\ell_1} \dots \sum_{k_d=0}^{\ell_d} m^{(\sum_{i=1}^d k_i)} (-m)^{\sum_{i=1}^d (p_i - \ell_i)} \prod_{i=1}^d \binom{p_i}{\ell_i} \brkbinom{\ell_i}{k_i} x_i^{p_i - \ell_i + k_i},
        \end{equation}
        where $\binom{p}{\ell}$ denotes the binomial coefficient $\frac{p!}{\ell! (p - \ell)!}$.
    \end{Theorem}

    \begin{proof}
        By applying the binomial formula to each factor $(\xi_i - \EE[\xi_i])^{p_i}$ and using the fact that $\EE[\xi_i] = m x_i$ for all $i\in \{1,2,\dots,d\}$, note that
        \begin{equation}
            \EE\Big[\prod_{i=1}^d (\xi_i - \EE[\xi_i])^{p_i}\Big] = \sum_{\ell_1=0}^{p_1} \dots \sum_{\ell_d=0}^{p_d} \EE\Big[\prod_{i=1}^d \xi_i^{\ell_i}\Big] \cdot \prod_{i=1}^d \binom{p_i}{\ell_i} (- m x_i)^{p_i - \ell_i}.
        \end{equation}
        The conclusion follows from Theorem~\ref{thm:non.central.moments}.
    \end{proof}

\section{Numerical implementation}

    The formulas in Theorem~\ref{thm:non.central.moments} and Theorem~\ref{thm:central.moments} can be implemented in \texttt{Mathematica} as follows:
    \begin{verbatim}
    NonCentral[m_, x_, p_, d_] :=
        Sum[FactorialPower[m, Sum[k[i], {i, 1, d}]] *
            Product[StirlingS2[p[[i]], k[i]] * x[[i]] ^ k[i], {i, 1, d}], ##] & @@
                ({k[#], 0, p[[#]]} & /@ Range[d]);
    Central[m_, x_, p_, d_] :=
        Sum[Sum[FactorialPower[m, Sum[k[i], {i, 1, d}]] *
            (-m) ^ Sum[p[[i]] - ell[i], {i, 1, d}] *
                Product[Binomial[p[[i]], ell[i]] * StirlingS2[ell[i], k[i]] *
                    x[[i]] ^ (p[[i]] - ell[i] + k[i]), {i, 1, d}], ##] & @@
                        ({k[#], 0, ell[#]} & /@ Range[d]), ##] & @@
                            ({ell[#], 0, p[[#]]} & /@ Range[d]);
    \end{verbatim}

\section{Explicit formulas}

    In \cite{Newcomer2008phd}, explicit expressions for the non-central moments of order~3 and 4 where obtained for the multinomial distribution, see also \cite{Newcomer_et_al_2008_tech_report,arXiv:2006.09059}.
    To expand on those results, we use the formula from Theorem~\ref{thm:non.central.moments} in the two subsections below to calculate (explicitly) all the non-central moments up to order~8 and all the central moments up to order~4.

    \vspace{3mm}
    Here is a table of the Stirling numbers of the second kind that we will use in our calculations:
    \small
    \begin{equation*}
        \begin{aligned}
            &\brkbinom{0}{0} = 1, \\[0.5mm]
            &\brkbinom{1}{0} = 0, ~ \brkbinom{1}{1} = 1, \\[0.5mm]
            &\brkbinom{2}{0} = 0, ~ \brkbinom{2}{1} = 1, \brkbinom{2}{2} = 1, \\[0.5mm]
            &\brkbinom{3}{0} = 0, ~ \brkbinom{3}{1} = 1, \brkbinom{3}{2} = 3, ~ \brkbinom{3}{3} = 1, \\[0.5mm]
            &\brkbinom{4}{0} = 0, ~ \brkbinom{4}{1} = 1, \brkbinom{4}{2} = 7, ~ \brkbinom{4}{3} = 6, \brkbinom{4}{4} = 1, \\[0.5mm]
            &\brkbinom{5}{0} = 0, ~ \brkbinom{5}{1} = 1, \brkbinom{5}{2} = 15, ~ \brkbinom{5}{3} = 25, \brkbinom{5}{4} = 10, ~ \brkbinom{5}{5} = 1, \\[0.5mm]
            &\brkbinom{6}{0} = 0, ~ \brkbinom{6}{1} = 1, \brkbinom{6}{2} = 31, ~ \brkbinom{6}{3} = 90, \brkbinom{6}{4} = 65, ~ \brkbinom{6}{5} = 15, \brkbinom{6}{6} = 1, \\[0.5mm]
            &\brkbinom{7}{0} = 0, ~ \brkbinom{7}{1} = 1, \brkbinom{7}{2} = 63, ~ \brkbinom{7}{3} = 301, \brkbinom{7}{4} = 350, ~ \brkbinom{7}{5} = 140, \brkbinom{7}{6} = 21, ~ \brkbinom{7}{7} = 1, \\[0.5mm]
            &\brkbinom{8}{0} = 0, ~ \brkbinom{8}{1} = 1, \brkbinom{8}{2} = 127, ~ \brkbinom{8}{3} = 966, \brkbinom{8}{4} = 1701, ~ \brkbinom{8}{5} = 1050, \brkbinom{8}{6} = 266, ~ \brkbinom{8}{7} = 28, ~ \brkbinom{8}{8} = 1.
        \end{aligned}
    \end{equation*}

    \subsection{Computation of the non-central moments up to order~8}

    By applying the general expression in Theorem~\ref{thm:non.central.moments} and by removing the Stirling numbers $\brkbinom{p_i}{k_i}$ that are equal to $0$, we get the following results directly.

    \vspace{3mm}
    \noindent
    \underline{Order $1$:} For any $j_1\in \{1,2,\dots,d\}$,
    \begin{align}
        \EE[\xi_{j_1}]
        &= x_{j_1} m.
    \end{align}

    \vspace{2mm}
    \noindent
    \underline{Order $2$:} For any distinct $j_1, j_2\in \{1,2,\dots,d\}$,
    \begin{align}
        \EE[\xi_{j_1}^2]
        &= x_{j_1} \big[m + m^{(2)} x_{j_1}\big], \\[1.1mm]
        \EE[\xi_{j_1} \xi_{j_2}]
        &= x_{j_1} x_{j_2} m^{(2)}.
    \end{align}

    \vspace{2mm}
    \noindent
    \underline{Order $3$:} For any distinct $j_1, j_2, j_3\in \{1,2,\dots,d\}$,
    \begin{align}
        \EE[\xi_{j_1}^3]
        &= x_{j_1} \big[m + 3 m^{(2)} x_{j_1} + m^{(3)} x_{j_1}^2\big], \\[1.1mm]
        \EE[\xi_{j_1}^2 \xi_{j_2}]
        &= x_{j_1} x_{j_2} \big[m^{(2)} + m^{(3)} x_{j_1}\big], \\[1.1mm]
        \EE[\xi_{j_1} \xi_{j_2} \xi_{j_3}]
        &= x_{j_1} x_{j_2} x_{j_3} m^{(3)}.
    \end{align}

    \vspace{2mm}
    \noindent
    \underline{Order $4$:} For any distinct $j_1, j_2, j_3, j_4\in \{1,2,\dots,d\}$,
    \begin{align}
        \EE[\xi_{j_1}^4]
        &= x_{j_1} \big[m + 7 m^{(2)} x_{j_1} + 6 m^{(3)} x_{j_1}^2 + m^{(4)} x_{j_1}^3\big], \\[1.1mm]
        \EE[\xi_{j_1}^3 \xi_{j_2}]
        &= x_{j_1} x_{j_2} \big[m^{(2)} + 3 m^{(3)} x_{j_1} + m^{(4)} x_{j_1}^2\big], \\[1.1mm]
        \EE[\xi_{j_1}^2 \xi_{j_2}^2]
        &= x_{j_1} x_{j_2} \big[m^{(2)} + m^{(3)} (x_{j_1} + x_{j_2}) + m^{(4)} x_{j_1} x_{j_2}\big], \\[1.1mm]
        \EE[\xi_{j_1}^2 \xi_{j_2} \xi_{j_3}]
        &= x_{j_1} x_{j_2} x_{j_3} \big[m^{(3)} + m^{(4)} x_{j_1}\big], \\[1.1mm]
        \EE[\xi_{j_1} \xi_{j_2} \xi_{j_3} \xi_{j_4}]
        &= x_{j_1} x_{j_2} x_{j_3} x_{j_4} m^{(4)}.
    \end{align}

    \vspace{2mm}
    \noindent
    \underline{Order $5$:} For any distinct $j_1, j_2, j_3, j_4, j_5\in \{1,2,\dots,d\}$,
    \begin{align}
        \EE[\xi_{j_1}^5]
        &= x_{j_1} \big[m + 15 m^{(2)} x_{j_1} + 25 m^{(3)} x_{j_1}^2 + 10 m^{(4)} x_{j_1}^3 + m^{(5)} x_{j_1}^4\big], \\[1.1mm]
        \EE[\xi_{j_1}^4 \xi_{j_2}]
        &= x_{j_1} x_{j_2} \big[m^{(2)} + 7 m^{(3)} x_{j_1} + 6 m^{(4)} x_{j_1}^2 + m^{(5)} x_{j_1}^3\big], \\[1.1mm]
        \EE[\xi_{j_1}^3 \xi_{j_2}^2]
        &= x_{j_1} x_{j_2} \big[m^{(2)} + m^{(3)} (3 x_{j_1} + x_{j_2}) + m^{(4)} (x_{j_1}^2 + 3 x_{j_1} x_{j_2}) + m^{(5)} x_{j_1}^2 x_{j_2}\big], \\[1.1mm]
        \EE[\xi_{j_1}^3 \xi_{j_2} \xi_{j_3}]
        &= x_{j_1} x_{j_2} x_{j_3} \big[m^{(3)} + 3 m^{(4)} x_{j_1} + m^{(5)} x_{j_1}^2\big], \\[1.1mm]
        \EE[\xi_{j_1}^2 \xi_{j_2}^2 \xi_{j_3}]
        &= x_{j_1} x_{j_2} x_{j_3} \big[m^{(3)} + m^{(4)} (x_{j_1} + x_{j_2}) + m^{(5)} x_{j_1} x_{j_2}\big], \\[1.1mm]
        \EE[\xi_{j_1}^2 \xi_{j_2} \xi_{j_3} \xi_{j_4}]
        &= x_{j_1} x_{j_2} x_{j_3} x_{j_4} \big[m^{(4)} + m^{(5)} x_{j_1}\big], \\[1.1mm]
        \EE[\xi_{j_1} \xi_{j_2} \xi_{j_3} \xi_{j_4} \xi_{j_5}]
        &= x_{j_1} x_{j_2} x_{j_3} x_{j_4} x_{j_5} m^{(5)}.
    \end{align}

    \vspace{2mm}
    \noindent
    \underline{Order $6$:} For any distinct $j_1, j_2, j_3, j_4, j_5, j_6\in \{1,2,\dots,d\}$,
    \begin{align}
        \EE[\xi_{j_1}^6]
        &= x_{j_1} \big[m + 31 m^{(2)} x_{j_1} + 90 m^{(3)} x_{j_1}^2 + 65 m^{(4)} x_{j_1}^3 + 15 m^{(5)} x_{j_1}^4 + m^{(6)} x_{j_1}^5\big], \\[1.1mm]
        \EE[\xi_{j_1}^5 \xi_{j_2}]
        &= x_{j_1} x_{j_2} \big[m^{(2)} + 15 m^{(3)} x_{j_1} + 25 m^{(4)} x_{j_1}^2 + 10 m^{(5)} x_{j_1}^3 + m^{(6)} x_{j_1}^4\big], \\[1.1mm]
        \EE[\xi_{j_1}^4 \xi_{j_2}^2]
        &= x_{j_1} x_{j_2} \left[\hspace{-1mm}
            \begin{array}{l}
                m^{(2)} + m^{(3)} (7 x_{j_1} + x_{j_2}) + m^{(4)} (6 x_{j_1}^2 + 7 x_{j_1} x_{j_2}) \\[0.49mm]
                + m^{(5)} (x_{j_1}^3 + 6 x_{j_1}^2 x_{j_2}) + m^{(6)} x_{j_1}^3 x_{j_2}
            \end{array}
            \hspace{-1mm}\right], \\[1.1mm]
        \EE[\xi_{j_1}^4 \xi_{j_2} \xi_{j_3}]
        &= x_{j_1} x_{j_2} x_{j_3} \big[m^{(3)} + 7 m^{(4)} x_{j_1} + 6 m^{(5)} x_{j_1}^2 + m^{(6)} x_{j_1}^3\big], \\[1.1mm]
        \EE[\xi_{j_1}^3 \xi_{j_2}^3]
        &= x_{j_1} x_{j_2} \left[\hspace{-1mm}
            \begin{array}{l}
                m^{(2)} + m^{(3)} (3 x_{j_1} + 3 x_{j_2}) + m^{(4)} (x_{j_1}^2 + 9 x_{j_1} x_{j_2} + x_{j_2}^2) \\[0.49mm]
                + m^{(5)} (3 x_{j_1}^2 x_{j_2} + 3 x_{j_1} x_{j_2}^2) + m^{(6)} x_{j_1}^2 x_{j_2}^2
            \end{array}
            \hspace{-1mm}\right], \\[1.1mm]
        \EE[\xi_{j_1}^3 \xi_{j_2}^2 \xi_{j_3}]
        &= x_{j_1} x_{j_2} x_{j_3} \big[m^{(3)} + m^{(4)} (3 x_{j_1} + x_{j_2}) + m^{(5)} (x_{j_1}^2 + 3 x_{j_1} x_{j_2}) + m^{(6)} x_{j_1}^2 x_{j_2}\big], \\[1.1mm]
        \EE[\xi_{j_1}^3 \xi_{j_2} \xi_{j_3} \xi_{j_4}]
        &= x_{j_1} x_{j_2} x_{j_3} x_{j_4} \big[m^{(4)} + 3 m^{(5)} x_{j_1} + m^{(6)} x_{j_1}^2\big], \\[1.1mm]
        \EE[\xi_{j_1}^2 \xi_{j_2}^2 \xi_{j_3}^2]
        &= x_{j_1} x_{j_2} x_{j_3} \left[\hspace{-1mm}
            \begin{array}{l}
                m^{(3)} + m^{(4)} (x_{j_1} + x_{j_2} + x_{j_3}) \\[0.49mm]
                + m^{(5)} (x_{j_1} x_{j_2} + x_{j_1} x_{j_3} + x_{j_2} x_{j_3}) + m^{(6)} x_{j_1} x_{j_2} x_{j_3}
            \end{array}
            \hspace{-1mm}\right], \\[1.1mm]
        \EE[\xi_{j_1}^2 \xi_{j_2}^2 \xi_{j_3} \xi_{j_4}]
        &= x_{j_1} x_{j_2} x_{j_3} x_{j_4} \big[m^{(4)} + m^{(5)} (x_{j_1} + x_{j_2}) + m^{(6)} x_{j_1} x_{j_2}\big], \\[1.1mm]
        \EE[\xi_{j_1}^2 \xi_{j_2} \xi_{j_3} \xi_{j_4} \xi_{j_5}]
        &= x_{j_1} x_{j_2} x_{j_3} x_{j_4} x_{j_5} \big[m^{(5)} + m^{(6)} x_{j_1}\big], \\[1.1mm]
        \EE[\xi_{j_1} \xi_{j_2} \xi_{j_3} \xi_{j_4} \xi_{j_5} \xi_{j_6}]
        &= x_{j_1} x_{j_2} x_{j_3} x_{j_4} x_{j_5} x_{j_6} m^{(6)}.
    \end{align}

    \vspace{2mm}
    \noindent
    \underline{Order $7$:} For any distinct $j_1, j_2, j_3, j_4, j_5, j_6, j_7\in \{1,2,\dots,d\}$,
    \begin{align}
        \EE[\xi_{j_1}^7]
        &= x_{j_1} \left[\hspace{-1mm}
            \begin{array}{l}
                m + 63 m^{(2)} x_{j_1} + 301 m^{(3)} x_{j_1}^2 + 350 m^{(4)} x_{j_1}^3 \\[0.49mm]
                + 140 m^{(5)} x_{j_1}^4 + 21 m^{(6)} x_{j_1}^5 + m^{(7)} x_{j_1}^6
            \end{array}
            \hspace{-1mm}\right], \\[1.1mm]
        \EE[\xi_{j_1}^6 \xi_{j_2}]
        &= x_{j_1} x_{j_2} \left[\hspace{-1mm}
            \begin{array}{l}
                m^{(2)} + 31 m^{(3)} x_{j_1} + 90 m^{(4)} x_{j_1}^2 \\[0.49mm]
                + 65 m^{(5)} x_{j_1}^3 + 15 m^{(6)} x_{j_1}^4 + m^{(7)} x_{j_1}^5
            \end{array}
            \hspace{-1mm}\right], \\[1.1mm]
        \EE[\xi_{j_1}^5 \xi_{j_2}^2]
        &= x_{j_1} x_{j_2} \left[\hspace{-1mm}
            \begin{array}{l}
                m^{(2)} + m^{(3)} (15 x_{j_1} + x_{j_2}) + m^{(4)} (25 x_{j_1}^2 + 15 x_{j_1} x_{j_2}) \\[0.49mm]
                + m^{(5)} (10 x_{j_1}^3 + 25 x_{j_1}^2 x_{j_2}) + m^{(6)} (x_{j_1}^4 + 10 x_{j_1}^3 x_{j_2}) + m^{(7)} x_{j_1}^4 x_{j_2}
            \end{array}
            \hspace{-1mm}\right], \\[1.1mm]
        \EE[\xi_{j_1}^5 \xi_{j_2} \xi_{j_3}]
        &= x_{j_1} x_{j_2} x_{j_3} \big[m^{(3)} + 15 m^{(4)} x_{j_1} + 25 m^{(5)} x_{j_1}^2 + 10 m^{(6)} x_{j_1}^3 + m^{(7)} x_{j_1}^4\big], \\[1.1mm]
        \EE[\xi_{j_1}^4 \xi_{j_2}^3]
        &= x_{j_1} x_{j_2} \left[\hspace{-1mm}
            \begin{array}{l}
                m^{(2)} + m^{(3)} (7 x_{j_1} + 3 x_{j_2}) + m^{(4)} (6 x_{j_1}^2 + 21 x_{j_1} x_{j_2} + x_{j_2}^2) \\[0.49mm]
                + m^{(5)} (x_{j_1}^3 + 18 x_{j_1}^2 x_{j_2} + 7 x_{j_1} x_{j_2}^2) \\[0.49mm]
                + m^{(6)} (3 x_{j_1}^3 x_{j_2} + 6 x_{j_1}^2 x_{j_2}^2) + m^{(7)} x_{j_1}^3 x_{j_2}^2
            \end{array}
            \hspace{-1mm}\right], \\[1.1mm]
        \EE[\xi_{j_1}^4 \xi_{j_2}^2 \xi_{j_3}]
        &= x_{j_1} x_{j_2} x_{j_3} \left[\hspace{-1mm}
            \begin{array}{l}
                m^{(3)} + m^{(4)} (7 x_{j_1} + x_{j_2}) + m^{(5)} (6 x_{j_1}^2 + 7 x_{j_1} x_{j_2}) \\[0.49mm]
                + m^{(6)} (x_{j_1}^3 + 6 x_{j_1}^2 x_{j_2}) + m^{(7)} x_{j_1}^3 x_{j_2}
            \end{array}
            \hspace{-1mm}\right], \\[1.1mm]
        \EE[\xi_{j_1}^4 \xi_{j_2} \xi_{j_3} \xi_{j_4}]
        &= x_{j_1} x_{j_2} x_{j_3} x_{j_4} \big[m^{(4)} + 7 m^{(5)} x_{j_1} + 6 m^{(6)} x_{j_1}^2 + m^{(7)} x_{j_1}^3\big], \\[1.1mm]
        \EE[\xi_{j_1}^3 \xi_{j_2}^3 \xi_{j_3}]
        &= x_{j_1} x_{j_2} x_{j_3} \left[\hspace{-1mm}
            \begin{array}{l}
                m^{(3)} + m^{(4)} (3 x_{j_1} + 3 x_{j_2}) + m^{(5)} (x_{j_1}^2 + 9 x_{j_1} x_{j_2} + x_{j_2}^2) \\[0.49mm]
                + m^{(6)} (3 x_{j_1}^2 x_{j_2} + 3 x_{j_1} x_{j_2}^2) + m^{(7)} x_{j_1}^2 x_{j_2}^2
            \end{array}
            \hspace{-1mm}\right], \\[1.1mm]
        \EE[\xi_{j_1}^3 \xi_{j_2}^2 \xi_{j_3}^2]
        &= x_{j_1} x_{j_2} x_{j_3} \left[\hspace{-1mm}
            \begin{array}{l}
                m^{(3)} + m^{(4)} (3 x_{j_1} + x_{j_2} + x_{j_3}) \\[0.49mm]
                + m^{(5)} (x_{j_1}^2 + 3 x_{j_1} x_{j_2} + 3 x_{j_1} x_{j_3} + x_{j_2} x_{j_3}) \\[0.49mm]
                + m^{(6)} (x_{j_1}^2 x_{j_2} + x_{j_1}^2 x_{j_3} + 3 x_{j_1} x_{j_2} x_{j_3}) + m^{(7)} x_{j_1}^2 x_{j_2} x_{j_3}
            \end{array}
            \hspace{-1mm}\right], \\[1.1mm]
        \EE[\xi_{j_1}^3 \xi_{j_2}^2 \xi_{j_3} \xi_{j_4}]
        &= x_{j_1} x_{j_2} x_{j_3} x_{j_4} \left[\hspace{-1mm}
            \begin{array}{l}
                m^{(4)} + m^{(5)} (3 x_{j_1} + x_{j_2}) \\[0.49mm]
                + m^{(6)} (x_{j_1}^2 + 3 x_{j_1} x_{j_2}) + m^{(7)} x_{j_1}^2 x_{j_2}
            \end{array}
            \hspace{-1mm}\right], \\[1.1mm]
        \EE[\xi_{j_1}^3 \xi_{j_2} \xi_{j_3} \xi_{j_4} \xi_{j_5}]
        &= x_{j_1} x_{j_2} x_{j_3} x_{j_4} x_{j_5} \big[m^{(5)} + 3 m^{(6)} x_{j_1} + m^{(7)} x_{j_1}^2\big], \\[1.1mm]
        \EE[\xi_{j_1}^2 \xi_{j_2}^2 \xi_{j_3}^2 \xi_{j_4}]
        &= x_{j_1} x_{j_2} x_{j_3} x_{j_4} \left[\hspace{-1mm}
            \begin{array}{l}
                m^{(4)} + m^{(5)} (x_{j_1} + x_{j_2} + x_{j_3}) \\[0.49mm]
                + m^{(6)} (x_{j_1} x_{j_2} + x_{j_1} x_{j_3} + x_{j_2} x_{j_3}) + m^{(7)} x_{j_1} x_{j_2} x_{j_3}
            \end{array}
            \hspace{-1mm}\right], \\[1.1mm]
        \EE[\xi_{j_1}^2 \xi_{j_2}^2 \xi_{j_3} \xi_{j_4} \xi_{j_5}]
        &= x_{j_1} x_{j_2} x_{j_3} x_{j_4} x_{j_5} \big[m^{(5)} + m^{(6)} (x_{j_1} + x_{j_2}) + m^{(7)} x_{j_1} x_{j_2}\big], \\[1.1mm]
        \EE[\xi_{j_1}^2 \xi_{j_2} \xi_{j_3} \xi_{j_4} \xi_{j_5} \xi_{j_6}]
        &= x_{j_1} x_{j_2} x_{j_3} x_{j_4} x_{j_5} x_{j_6} \big[m^{(6)} + m^{(7)} x_{j_1}\big], \\[1.1mm]
        \EE[\xi_{j_1} \xi_{j_2} \xi_{j_3} \xi_{j_4} \xi_{j_5} \xi_{j_6} \xi_{j_7}]
        &= x_{j_1} x_{j_2} x_{j_3} x_{j_4} x_{j_5} x_{j_6} x_{j_7} m^{(7)}.
    \end{align}

    \vspace{2mm}
    \noindent
    \underline{Order $8$:} For any distinct $j_1, j_2, j_3, j_4, j_5, j_6, j_7, j_8\in \{1,2,\dots,d\}$,
    \begin{align}
        \EE[\xi_{j_1}^8]
        &= x_{j_1} \left[\hspace{-1mm}
            \begin{array}{l}
                m + 127 m^{(2)} x_{j_1} + 966 m^{(3)} x_{j_2}^2 + 1701 m^{(4)} x_{j_1}^3 \\[0.49mm]
                + 1050 m^{(5)} x_{j_1}^4 + 266 m^{(6)} x_{j_1}^5 + 28 m^{(7)} x_{j_1}^6 + m^{(8)} x_{j_1}^7
            \end{array}
            \hspace{-1mm}\right], \\[1.1mm]
        \EE[\xi_{j_1}^7 \xi_{j_2}]
        &= x_{j_1} x_{j_2} \left[\hspace{-1mm}
            \begin{array}{l}
                m + 63 m^{(3)} x_{j_1} + 301 m^{(4)} x_{j_1}^2 + 350 m^{(5)} x_{j_1}^3 \\[0.49mm]
                + 140 m^{(6)} x_{j_1}^4 + 21 m^{(7)} x_{j_1}^5 + m^{(8)} x_{j_1}^6
            \end{array}
            \hspace{-1mm}\right], \\[1.1mm]
        \EE[\xi_{j_1}^6 \xi_{j_2}^2]
        &= x_{j_1} x_{j_2} \left[\hspace{-1mm}
            \begin{array}{l}
                m^{(2)} + m^{(3)} (31 x_{j_1} + x_{j_2}) + m^{(4)} (90 x_{j_1}^2 + 31 x_{j_1} x_{j_2}) \\[0.49mm]
                + m^{(5)} (65 x_{j_1}^3 + 90 x_{j_1}^2 x_{j_2}) + m^{(6)} (15 x_{j_1}^4 + 65 x_{j_1}^3 x_{j_2}) \\[0.49mm]
                + m^{(7)} (x_{j_1}^5 + 15 x_{j_1}^4 x_{j_2}) + m^{(8)} x_{j_1}^5 x_{j_2}
            \end{array}
            \hspace{-1mm}\right], \\[1.1mm]
        \EE[\xi_{j_1}^6 \xi_{j_2} \xi_{j_3}]
        &= x_{j_1} x_{j_2} x_{j_3} \left[\hspace{-1mm}
            \begin{array}{l}
                m^{(3)} + 31 m^{(4)} x_{j_1} + 90 m^{(5)} x_{j_1}^2 \\[0.49mm]
                + 65 m^{(6)} x_{j_1}^3 + 15 m^{(7)} x_{j_1}^4 + m^{(8)} x_{j_1}^5
            \end{array}
            \hspace{-1mm}\right], \\[1.1mm]
        \EE[\xi_{j_1}^5 \xi_{j_2}^3]
        &= x_{j_1} x_{j_2} \left[\hspace{-1mm}
            \begin{array}{l}
                m^{(2)} + m^{(3)} (15 x_{j_1} + 3 x_{j_2}) + m^{(4)} (25 x_{j_1}^2 + 45 x_{j_1} x_{j_2} + x_{j_2}^2) \\[0.49mm]
                + m^{(5)} (10 x_{j_1}^3 + 75 x_{j_1}^2 x_{j_2} + 15 x_{j_1} x_{j_2}^2) \\[0.49mm]
                + m^{(6)} (x_{j_4}^4 + 30 x_{j_1}^3 x_{j_2} + 25 x_{j_1}^2 x_{j_2}^2) \\[0.49mm]
                + m^{(7)} (3 x_{j_1}^4 x_{j_2} + 10 x_{j_1}^3 x_{j_2}^2) + m^{(8)} x_{j_1}^4 x_{j_2}^2
            \end{array}
            \hspace{-1mm}\right], \\[1.1mm]
        \EE[\xi_{j_1}^5 \xi_{j_2}^2 \xi_{j_3}]
        &= x_{j_1} x_{j_2} x_{j_3} \left[\hspace{-1mm}
            \begin{array}{l}
                m^{(3)} + m^{(4)} (15 x_{j_1} + x_{j_2}) + m^{(5)} (25 x_{j_1}^2 + 15 x_{j_1} x_{j_2}) \\[0.49mm]
                + m^{(6)} (10 x_{j_1}^3 + 25 x_{j_1}^2 x_{j_2}) + m^{(7)} (x_{j_1}^4 + 10 x_{j_1}^3 x_{j_2}) + m^{(8)} x_{j_1}^4 x_{j_2}
            \end{array}
            \hspace{-1mm}\right], \\[1.1mm]
        \EE[\xi_{j_1}^5 \xi_{j_2} \xi_{j_3} \xi_{j_4}]
        &= x_{j_1} x_{j_2} x_{j_3} x_{j_4} \big[m^{(4)} + 15 m^{(5)} x_{j_1} + 25 m^{(6)} x_{j_1}^2 + 10 m^{(7)} x_{j_1}^3 + m^{(8)} x_{j_1}^4\big], \\[1.1mm]
        \EE[\xi_{j_1}^4 \xi_{j_2}^4]
        &= x_{j_1} x_{j_2} \left[\hspace{-1mm}
            \begin{array}{l}
                m^{(2)} + m^{(3)} (7 x_{j_1} + 7 x_{j_2}) + m^{(4)} (6 x_{j_1}^2 + 49 x_{j_1} x_{j_2} + 6 x_{j_2}^2) \\[0.49mm]
                + m^{(5)} (x_{j_1}^3 + 42 x_{j_1}^2 x_{j_2} + 42 x_{j_1} x_{j_2}^2 + x_{j_2}^3) \\[0.49mm]
                + m^{(6)} (7 x_{j_1}^3 x_{j_2} + 36 x_{j_1}^2 x_{j_2}^2 + 7 x_{j_1} x_{j_2}^3) \\[0.49mm]
                + m^{(7)} (6 x_{j_1}^3 x_{j_2}^2 + 6 x_{j_1}^2 x_{j_2}^3) + m^{(8)} x_{j_1}^3 x_{j_2}^3
            \end{array}
            \hspace{-1mm}\right], \\[1.1mm]
        \EE[\xi_{j_1}^4 \xi_{j_2}^3 \xi_{j_3}]
        &= x_{j_1} x_{j_2} x_{j_3} \left[\hspace{-1mm}
            \begin{array}{l}
                m^{(3)} + m^{(4)} (7 x_{j_1} + 3 x_{j_2}) + m^{(5)} (6 x_{j_1}^2 + 21 x_{j_1} x_{j_2} + x_{j_2}^2) \\[0.49mm]
                + m^{(6)} (x_{j_1}^3 + 18 x_{j_1}^2 x_{j_2} + 7 x_{j_1} x_{j_2}^2) \\[0.49mm]
                + m^{(7)} (3 x_{j_1}^3 x_{j_2} + 6 x_{j_1}^2 x_{j_2}^2) + m^{(8)} x_{j_1}^3 x_{j_2}^2
            \end{array}
            \hspace{-1mm}\right], \\[1.1mm]
        \EE[\xi_{j_1}^4 \xi_{j_2}^2 \xi_{j_3}^2]
        &= x_{j_1} x_{j_2} x_{j_3} \left[\hspace{-1mm}
            \begin{array}{l}
                m^{(3)} + m^{(4)} (7 x_{j_1} + x_{j_2} + x_{j_3}) \\[0.49mm]
                + m^{(5)} (6 x_{j_1}^2 + 7 x_{j_1} x_{j_2} + 7 x_{j_1} x_{j_3} + x_{j_2} x_{j_3}) \\[0.49mm]
                + m^{(6)} (x_{j_1}^3 + 6 x_{j_1}^2 x_{j_2} + 6 x_{j_1}^2 x_{j_3} + 7 x_{j_1} x_{j_2} x_{j_3}) \\[0.49mm]
                + m^{(7)} (x_{j_1}^3 x_{j_2} + x_{j_1}^3 x_{j_3} + 6 x_{j_1}^2 x_{j_2} x_{j_3}) + m^{(8)} x_{j_1}^3 x_{j_2} x_{j_3}
            \end{array}
            \hspace{-1mm}\right], \\[1.1mm]
        \EE[\xi_{j_1}^4 \xi_{j_2}^2 \xi_{j_3} \xi_{j_4}]
        &= x_{j_1} x_{j_2} x_{j_3} x_{j_4} \left[\hspace{-1mm}
            \begin{array}{l}
                m^{(4)} + m^{(5)} (7 x_{j_1} + x_{j_2}) + m^{(6)} (6 x_{j_1}^2 + 7 x_{j_1} x_{j_2}) \\[0.49mm]
                + m^{(7)} (x_{j_1}^3 + 6 x_{j_1}^2 x_{j_2}) + m^{(8)} x_{j_1}^3 x_{j_2}
            \end{array}
            \hspace{-1mm}\right], \\[1.1mm]
        \EE[\xi_{j_1}^4 \xi_{j_2} \xi_{j_3} \xi_{j_4} \xi_{j_5}]
        &= x_{j_1} x_{j_2} x_{j_3} x_{j_4} x_{j_5} \big[m^{(5)} + 7 m^{(6)} x_{j_1} + 6 m^{(7)} x_{j_1}^2 + m^{(8)} x_{j_1}^3\big], \\[1.1mm]
        \EE[\xi_{j_1}^3 \xi_{j_2}^3 \xi_{j_3}^2]
        &= x_{j_1} x_{j_2} x_{j_3} \left[\hspace{-1mm}
            \begin{array}{l}
                m^{(3)} + m^{(4)} (3 x_{j_1} + 3 x_{j_2} + x_{j_3}) \\[0.49mm]
                + m^{(5)} (x_{j_1}^2 + x_{j_2}^2 + 3 x_{j_1} x_{j_3} + 3 x_{j_2} x_{j_3} + 9 x_{j_1} x_{j_2}) \\[0.49mm]
                + m^{(6)} (x_{j_1}^2 x_{j_3} + x_{j_2}^2 x_{j_3} + 3 x_{j_1}^2 x_{j_2} + 3 x_{j_1} x_{j_2}^2 + 9 x_{j_1} x_{j_2} x_{j_3}) \\[0.49mm]
                + m^{(7)} (x_{j_1}^2 x_{j_2}^2 + 3 x_{j_1}^2 x_{j_2} x_{j_3} + 3 x_{j_1} x_{j_2}^2 x_{j_3}) + m^{(8)} x_{j_1}^2 x_{j_2}^2 x_{j_3}
            \end{array}
            \hspace{-1mm}\right], \\[1.1mm]
        \EE[\xi_{j_1}^3 \xi_{j_2}^3 \xi_{j_3} \xi_{j_4}]
        &= x_{j_1} x_{j_2} x_{j_3} x_{j_4} \left[\hspace{-1mm}
            \begin{array}{l}
                m^{(4)} + m^{(5)} (3 x_{j_1} + 3 x_{j_2}) + m^{(6)} (x_{j_1}^2 + 9 x_{j_1} x_{j_2} + x_{j_2}^2) \\[0.49mm]
                + m^{(7)} (3 x_{j_1}^2 x_{j_2} + 3 x_{j_1} x_{j_2}^2) + m^{(8)} x_{j_1}^2 x_{j_2}^2
            \end{array}
            \hspace{-1mm}\right], \\[1.1mm]
        \EE[\xi_{j_1}^3 \xi_{j_2}^2 \xi_{j_3}^2 \xi_{j_4}]
        &= x_{j_1} x_{j_2} x_{j_3} x_{j_4} \left[\hspace{-1mm}
            \begin{array}{l}
                m^{(4)} + m^{(5)} (3 x_{j_1} + x_{j_2} + x_{j_3}) \\[0.49mm]
                + m^{(6)} (3 x_{j_1} x_{j_2} + 3 x_{j_1} x_{j_3} + x_{j_2} x_{j_3}) \\[0.49mm]
                + m^{(7)} (x_{j_1}^2 x_{j_2} + x_{j_1}^2 x_{j_3} + 3 x_{j_1} x_{j_2} x_{j_3}) + m^{(8)} x_{j_1}^2 x_{j_2} x_{j_3}
            \end{array}
            \hspace{-1mm}\right], \\[1.1mm]
        \EE[\xi_{j_1}^3 \xi_{j_2}^2 \xi_{j_3} \xi_{j_4} \xi_{j_5}]
        &= x_{j_1} x_{j_2} x_{j_3} x_{j_4} x_{j_5} \left[\hspace{-1mm}
            \begin{array}{l}
                m^{(5)} + m^{(6)} (3 x_{j_1} + x_{j_2}) \\[0.49mm]
                + m^{(7)} (x_{j_1}^2 + 3 x_{j_1} x_{j_2}) + m^{(8)} x_{j_1}^2 x_{j_2}
            \end{array}
            \hspace{-1mm}\right], \\[1.1mm]
        \EE[\xi_{j_1}^3 \xi_{j_2} \xi_{j_3} \xi_{j_4} \xi_{j_5} \xi_{j_6}]
        &= x_{j_1} x_{j_2} x_{j_3} x_{j_4} x_{j_5} x_{j_6} \big[m^{(6)} + 3 m^{(7)} x_{j_1} + m^{(8)} x_{j_1}^2\big], \\[1.1mm]
        \EE[\xi_{j_1}^2 \xi_{j_2}^2 \xi_{j_3}^2 \xi_{j_4}^2]
        &= x_{j_1} x_{j_2} x_{j_3} x_{j_4} \left[\hspace{-1mm}
            \begin{array}{l}
                m^{(4)} + m^{(5)} (x_{j_1} + x_{j_2} + x_{j_3} + x_{j_4}) \\[0.49mm]
                + m^{(6)} (x_{j_1} x_{j_2} + x_{j_1} x_{j_3} + x_{j_1} x_{j_4} + x_{j_2} x_{j_3} + x_{j_2} x_{j_4} + x_{j_3} x_{j_4}) \\[0.49mm]
                + m^{(7)} (x_{j_1} x_{j_2} x_{j_3} + x_{j_1} x_{j_2} x_{j_4} + x_{j_1} x_{j_3} x_{j_4} + x_{j_2} x_{j_3} x_{j_4}) \\[0.49mm]
                + m^{(8)} x_{j_1} x_{j_2} x_{j_3} x_{j_4}
            \end{array}
            \hspace{-1mm}\right], \\[-3.5mm]
        \EE[\xi_{j_1}^2 \xi_{j_2}^2 \xi_{j_3}^2 \xi_{j_4} \xi_{j_5}]
        &= x_{j_1} x_{j_2} x_{j_3} x_{j_4} x_{j_5} \left[\hspace{-1mm}
            \begin{array}{l}
                m^{(5)} + m^{(6)} (x_{j_1} + x_{j_2} + x_{j_3}) \\[0.49mm]
                + m^{(7)} (x_{j_1} x_{j_2} + x_{j_1} x_{j_3} + x_{j_2} x_{j_3}) + m^{(8)} x_{j_1} x_{j_2} x_{j_3}
            \end{array}
            \hspace{-1mm}\right], \\[1.1mm]
        \EE[\xi_{j_1}^2 \xi_{j_2}^2 \xi_{j_3} \xi_{j_4} \xi_{j_5} \xi_{j_6}]
        &= x_{j_1} x_{j_2} x_{j_3} x_{j_4} x_{j_5} x_{j_6} \big[m^{(6)} + m^{(7)} (x_{j_1} + x_{j_2}) + m^{(8)} x_{j_1} x_{j_2}\big], \\[1.1mm]
        \EE[\xi_{j_1}^2 \xi_{j_2} \xi_{j_3} \xi_{j_4} \xi_{j_5} \xi_{j_6} \xi_{j_7}]
        &= x_{j_1} x_{j_2} x_{j_3} x_{j_4} x_{j_5} x_{j_6} x_{j_7} \big[m^{(7)} + m^{(8)} x_{j_1}\big], \\[1.1mm]
        \EE[\xi_{j_1} \xi_{j_2} \xi_{j_3} \xi_{j_4} \xi_{j_5} \xi_{j_6} \xi_{j_7} \xi_{j_8}]
        &= x_{j_1} x_{j_2} x_{j_3} x_{j_4} x_{j_5} x_{j_6} x_{j_7} x_{j_8} m^{(8)}.
    \end{align}

    \subsection{Computation of the central moments up to order~4}

    With the results of the previous subsection and some algebraic manipulations (or the formula in Theorem~\ref{thm:central.moments}), we can now calculate the central moments explicitly.
    We could calculate them up to order~8, but it would be very tedious.
    Instead, we write them up to order~4 for the sake of brevity.
    The simplifications we make to obtain the boxed expressions below are done with \texttt{Mathematica}.

    \iffalse
    \noindent
    \underline{Order $1$:} For any $j_1\in \{1,2,\dots,d\}$,
    \begin{align}
        \EE[\xi_{j_1} - \EE[\xi_{j_1}]] = \boxed{0}\, .
    \end{align}
    \fi

    \vspace{3mm}
    \noindent
    \underline{Order $2$:} For any distinct $j_1, j_2\in \{1,2,\dots,d\}$,
    \begin{align}
        \EE[(\xi_{j_1} - \EE[\xi_{j_1}])^2]
        &= \EE[\xi_{j_1}^2] - (\EE[\xi_{j_1}])^2
        = x_{j_1} \big[m + m^{(2)} x_{j_1}\big] - m^2 x_{j_1}^2 \notag \\[1.1mm]
        &= \boxed{m x_{j_1} (1 - x_{j_1})} \\[3.1mm]
        \EE[(\xi_{j_1} - \EE[\xi_{j_1}]) (\xi_{j_2} - \EE[\xi_{j_2}])]
        &= \EE[\xi_{j_1} \xi_{j_2}] - \EE[\xi_{j_1}] \EE[\xi_{j_2}]
        = m^{(2)} x_{j_1} x_{j_2} - m x_{j_1} m x_{j_2} \notag \\[1.1mm]
        &= \boxed{- m x_{j_1} x_{j_2}} \, .
    \end{align}

    \vspace{2mm}
    \noindent
    \underline{Order $3$:} For any distinct $j_1, j_2, j_3\in \{1,2,\dots,d\}$,
    \begin{align}
        \EE[(\xi_{j_1} - \EE[\xi_{j_1}])^3]
        &= \EE[\xi_{j_1}^3] - 3 \, \EE[\xi_{j_1}^2] \EE[\xi_{j_1}] + 2 \, (\EE[\xi_{j_1}])^3 \notag \\[1.1mm]
        &= x_{j_1} \big[m + 3 m^{(2)} x_{j_1} + m^{(3)} x_{j_1}^2\big] - 3 x_{j_1} \big[m + m^{(2)} x_{j_1}\big] m x_{j_1} + 2 m^3 x_{j_1}^3 \notag \\[1.1mm]
        &= \boxed{m x_{j_1} (x_{j_1} - 1) (2 x_{j_1} - 1)} \\[3.1mm]
        &\hspace{-23.6mm}\EE[(\xi_{j_1} - \EE[\xi_{j_1}])^2 (\xi_{j_2} - \EE[\xi_{j_2}])] \notag \\[1.1mm]
        &= \EE[\xi_{j_1}^2 \xi_{j_2}] - \EE[\xi_{j_1}^2] \EE[\xi_{j_2}] - 2 \, \EE[\xi_{j_1} \xi_{j_2}] \EE[\xi_{j_1}] + 2 \, (\EE[\xi_{j_1}])^2 \EE[\xi_{j_2}] \notag \\[1.1mm]
        &= x_{j_1} x_{j_2} \big[m^{(2)} + m^{(3)} x_{j_1}\big] - x_{j_1} \big[m + m^{(2)} x_{j_1}\big] m x_{j_2} - 2 m^{(2)} x_{j_1} x_{j_2} m x_{j_1} + 2 m^2 x_{j_1}^2 m x_{j_2} \notag \\[1.1mm]
        &= \boxed{m x_{j_1} x_{j_2} (2 x_{j_1} - 1)} \\[3.1mm]
        &\hspace{-23.6mm}\EE[(\xi_{j_1} - \EE[\xi_{j_1}]) (\xi_{j_2} - \EE[\xi_{j_2}]) (\xi_{j_3} - \EE[\xi_{j_3}])] \notag \\[1.1mm]
        &= \EE[\xi_{j_1} \xi_{j_2} \xi_{j_3}] - \EE[\xi_{j_1} \xi_{j_2}] \EE[\xi_{j_3}] - \EE[\xi_{j_1} \xi_{j_3}] \EE[\xi_{j_2}] - \EE[\xi_{j_2} \xi_{j_3}] \EE[\xi_{j_1}] + 2 \, \EE[\xi_{j_1}] \EE[\xi_{j_2}] \EE[\xi_{j_3}] \notag \\[1.1mm]
        &= m^{(3)} x_{j_1} x_{j_2} x_{j_3} - m^{(2)} x_{j_1} x_{j_2} m x_{j_3} - m^{(2)} x_{j_1} x_{j_3} m x_{j_2} - m^{(2)} x_{j_2} x_{j_3} m x_{j_1} + 2 m^3 x_{j_1} x_{j_2} x_{j_3} \notag \\[1.1mm]
        &= \boxed{2 m x_{j_1} x_{j_2} x_{j_3}} \, .
    \end{align}

    \vspace{2mm}
    \noindent
    \underline{Order $4$:} For any distinct $j_1, j_2, j_3, j_4\in \{1,2,\dots,d\}$,
    \begin{align}
        &\EE[(\xi_{j_1} - \EE[\xi_{j_1}])^4] \notag \\[1.1mm]
        &\qquad= \EE[\xi_{j_1}^4] - 4 \, \EE[\xi_{j_1}^3] \EE[\xi_{j_1}] + 6 \, \EE[\xi_{j_1}^2] (\EE[\xi_{j_1}])^2 - 3 \, (\EE[\xi_{j_1}])^4 \notag \\[1.1mm]
        &\qquad= x_{j_1} \big[m + 7 m^{(2)} x_{j_1} + 6 m^{(3)} x_{j_1}^2 + m^{(4)} x_{j_1}^3\big] - 4 x_{j_1} \big[m + 3 m^{(2)} x_{j_1} + m^{(3)} x_{j_1}^2\big] m x_{j_1} \notag \\
        &\qquad\quad+ 6 x_{j_1} \big[m + m^{(2)} x_{j_1}\big] (m x_{j_1})^2 - 3 m^4 x_{j_1}^4 \notag \\
        &\qquad= \boxed{3 m^2 x_{j_1}^2 (x_{j_1} - 1)^2 + m x_{j_1} (1 - x_{j_1}) (6 x_{j_1}^2 - 6 x_{j_1} + 1)} \\[2.1mm]
        %%%
        &\EE[(\xi_{j_1} - \EE[\xi_{j_1}])^3 (\xi_{j_2} - \EE[\xi_{j_2}])] \notag \\[1.1mm]
        &\qquad= \EE[\xi_{j_1}^3 \xi_{j_2}] - \EE[\xi_{j_1}^3] \EE[\xi_{j_2}] - 3 \, \EE[\xi_{j_1}^2 \xi_{j_2}] \EE[\xi_{j_1}] + 3 \, \EE[\xi_{j_1}^2] \EE[\xi_{j_1}] \EE[\xi_{j_2}] \notag \\[1.1mm]
        &\qquad\quad+ 3 \, \EE[\xi_{j_1} \xi_{j_2}] (\EE[\xi_{j_1}])^2 - 3 \, (\EE[\xi_{j_1}])^3 \EE[\xi_{j_2}] \notag \\[1.1mm]
        &\qquad= x_{j_1} x_{j_2} \big[m^{(2)} + 3 m^{(3)} x_{j_1} + m^{(4)} x_{j_1}^2\big] - x_{j_1} \big[m + 3 m^{(2)} x_{j_1} + m^{(3)} x_{j_1}^2\big] m x_{j_2} \notag \\
        &\qquad\quad- 3 x_{j_1} x_{j_2} \big[m^{(2)} + m^{(3)} x_{j_1}\big] m x_{j_1} + 3 x_{j_1} \big[m + m^{(2)} x_{j_1}\big] m x_{j_1} m x_{j_2} + 3 m^{(2)} x_{j_1} x_{j_2} m^2 x_{j_1}^2 - 3 m^3 x_{j_1}^3 m x_{j_2} \notag \\
        &\qquad= \boxed{m x_{j_1} x_{j_2} (3 (m - 2) x_{j_1} (x_{j_1} - 1) - 1)} \\[2.1mm]
        %%%
        &\EE[(\xi_{j_1} - \EE[\xi_{j_1}])^2 (\xi_{j_2} - \EE[\xi_{j_2}])^2] \notag \\[1.1mm]
        &\qquad= \EE[\xi_{j_1}^2 \xi_{j_2}^2] - 2 \, \EE[\xi_{j_1}^2 \xi_{j_2}] \EE[\xi_{j_2}] - 2 \, \EE[\xi_{j_1} \xi_{j_2}^2] \EE[\xi_{j_1}] + \EE[\xi_{j_1}^2] (\EE[\xi_{j_2}])^2 + \EE[\xi_{j_2}^2] (\EE[\xi_{j_1}])^2 \notag \\[1.1mm]
        &\qquad\quad+ 4 \, \EE[\xi_{j_1} \xi_{j_2}] \EE[\xi_{j_1}] \EE[\xi_{j_2}] - 3 \, (\EE[\xi_{j_1}])^2 (\EE[\xi_{j_2}])^2 \notag \\[1.1mm]
        &\qquad= x_{j_1} x_{j_2} \big[m^{(2)} + m^{(3)} (x_{j_1} + x_{j_2}) + m^{(4)} x_{j_1} x_{j_2}\big] - 2 x_{j_1} x_{j_2} \big[m^{(2)} + m^{(3)} x_{j_1}\big] m x_{j_2} \notag \\
        &\qquad\quad- 2 x_{j_1} x_{j_2} \big[m^{(2)} + m^{(3)} x_{j_2}\big] m x_{j_1} + x_{j_1} \big[m + m^{(2)} x_{j_1}\big] m^2 x_{j_2}^2 + x_{j_2} \big[m + m^{(2)} x_{j_2}\big] m^2 x_{j_1}^2 \notag \\
        &\qquad\quad+ 4 m^{(2)} x_{j_1} x_{j_2} m x_{j_1} m x_{j_2} - 3 \, m^2 x_{j_1}^2 m^2 x_{j_2}^2 \notag \\
        &\qquad= \boxed{m (m - 2) x_{j_1} x_{j_2} (3 x_{j_1} x_{j_2} - (x_{j_1} + x_{j_2}) + 1) + m x_{j_1} x_{j_2}} \\[2.1mm]
        %%%
        &\EE[(\xi_{j_1} - \EE[\xi_{j_1}])^2 (\xi_{j_2} - \EE[\xi_{j_2}]) (\xi_{j_3} - \EE[\xi_{j_3}])] \notag \\[1.1mm]
        &\qquad= \EE[\xi_{j_1}^2 \xi_{j_2} \xi_{j_3}] - \EE[\xi_{j_1}^2 \xi_{j_2}] \EE[\xi_{j_3}] - \EE[\xi_{j_1}^2 \xi_{j_3}] \EE[\xi_{j_2}] - 2 \, \EE[\xi_{j_1} \xi_{j_2} \xi_{j_3}] \EE[\xi_{j_1}] + \EE[\xi_{j_1}^2] \EE[\xi_{j_2}] \EE[\xi_{j_3}] \notag \\[1.1mm]
        &\qquad\quad+ 2 \, \EE[\xi_{j_1} \xi_{j_2}] \EE[\xi_{j_1}] \EE[\xi_{j_3}] + 2 \, \EE[\xi_{j_1} \xi_{j_3}] \EE[\xi_{j_1}] \EE[\xi_{j_2}] + \EE[\xi_{j_2} \xi_{j_3}] (\EE[\xi_{j_1}])^2 - 3 \, (\EE[\xi_{j_1}])^2 \EE[\xi_{j_2}] \EE[\xi_{j_3}] \notag \\[1.1mm]
        &\qquad= x_{j_1} x_{j_2} x_{j_3} \big[m^{(3)} + m^{(4)} x_{j_1}\big] - x_{j_1} x_{j_2} \big[m^{(2)} + m^{(3)} x_{j_1}] m x_{j_3} - x_{j_1} x_{j_3} \big[m^{(2)} + m^{(3)} x_{j_1}] m x_{j_2} \notag \\
        &\qquad\quad- 2 m^{(3)} x_{j_1} x_{j_2} x_{j_3} m x_{j_1} + x_{j_1} \big[m + m^{(2)} x_{j_1}\big] m x_{j_2} m x_{j_3} + 2 m^{(2)} x_{j_1} x_{j_2} m x_{j_1} m x_{j_3}  \notag \\
        &\qquad\quad+ 2 m^{(2)} x_{j_1} x_{j_3} m x_{j_1} m x_{j_2} + m^{(2)} x_{j_2} x_{j_3} m^2 x_{j_1}^2 - 3 m^2 x_{j_1}^2 m x_{j_2} m x_{j_3} \notag \\
        &\qquad= \boxed{m (m - 2) x_{j_1} x_{j_2} x_{j_3} (3 x_{j_1} - 1)} \\[2.1mm]
        %%%
        &\EE[(\xi_{j_1} - \EE[\xi_{j_1}]) (\xi_{j_2} - \EE[\xi_{j_2}]) (\xi_{j_3} - \EE[\xi_{j_3}]) (\xi_{j_4} - \EE[\xi_{j_4}])] \notag \\[1.1mm]
        &\qquad= \EE[\xi_{j_1} \xi_{j_2} \xi_{j_3} \xi_{j_4}] - \EE[\xi_{j_1} \xi_{j_2} \xi_{j_3}] \EE[\xi_{j_4}] - \EE[\xi_{j_1} \xi_{j_2} \xi_{j_4}] \EE[\xi_{j_3}] - \EE[\xi_{j_1} \xi_{j_3} \xi_{j_4}] \EE[\xi_{j_2}] - \EE[\xi_{j_2} \xi_{j_3} \xi_{j_4}] \EE[\xi_{j_1}] \notag \\[1.1mm]
        &\qquad\quad+ \EE[\xi_{j_1} \xi_{j_2}] \EE[\xi_{j_3}] \EE[\xi_{j_4}] + \EE[\xi_{j_1} \xi_{j_3}] \EE[\xi_{j_2}] \EE[\xi_{j_4}] + \EE[\xi_{j_1} \xi_{j_4}] \EE[\xi_{j_2}] \EE[\xi_{j_3}] \notag \\[1.1mm]
        &\qquad\quad+ \EE[\xi_{j_2} \xi_{j_3}] \EE[\xi_{j_1}] \EE[\xi_{j_4}] + \EE[\xi_{j_2} \xi_{j_4}] \EE[\xi_{j_1}] \EE[\xi_{j_3}] + \EE[\xi_{j_3} \xi_{j_4}] \EE[\xi_{j_1}] \EE[\xi_{j_2}] - 3\, \EE[\xi_{j_1}] \EE[\xi_{j_2}] \EE[\xi_{j_3}] \EE[\xi_{j_4}] \notag \\[1.1mm]
        &\qquad= m^{(4)} x_{j_1} x_{j_2} x_{j_3} x_{j_4} - m^{(3)} x_{j_1} x_{j_2} x_{j_3} m x_{j_4} - m^{(3)} x_{j_1} x_{j_2} x_{j_4} m x_{j_3} - m^{(3)} x_{j_1} x_{j_3} x_{j_4} m x_{j_2} - m^{(3)} x_{j_2} x_{j_3} x_{j_4} m x_{j_1} \notag \\
        &\qquad\quad+ m^{(2)} x_{j_1} x_{j_2} m x_{j_3} m x_{j_4} + m^{(2)} x_{j_1} x_{j_3} m x_{j_2} m x_{j_4} + m^{(2)} x_{j_1} x_{j_4} m x_{j_2} m x_{j_3} \notag \\
        &\qquad\quad+ m^{(2)} x_{j_2} x_{j_3} m x_{j_1} m x_{j_4} + m^{(2)} x_{j_2} x_{j_4} m x_{j_1} m x_{j_3} + m^{(2)} x_{j_3} x_{j_4} m x_{j_1} m x_{j_2} - 3\, m^4 x_{j_1} x_{j_2} x_{j_3} x_{j_4} \notag \\
        &\qquad= \boxed{3 m (m - 2) x_{j_1} x_{j_2} x_{j_3} x_{j_4}}\, .
    \end{align}

\section{Conclusion}

    In this short paper, we found general formulas for the central and non-central moments of the multinomial distribution as well as explicit formulas for all the non-central moments up to order~8 and all the central moments up to order~4.
    Our work expands on the results in \cite{Newcomer2008phd}, where the central moments were calculated up to order~4.
    It also complements the general formula for the (joint) factorial moments from \cite{MR143299} and the explicit formulas for some of the lower-order (mixed) cumulants that were presented in \cite{MR33996}.

%
% ----------  B I B L I O G R A P H I E  ----------
%

\nocite{MR3825458}
\nocite{Ouimet2019phd}

\bibliographystyle{authordate1}
\bibliography{Ouimet_2020_moments_multinomial_arxiv_bib}

\end{document}